\documentclass{article}

\usepackage{PRIMEarxiv}
\usepackage{pgffor}
\usepackage[utf8]{inputenc} 
\usepackage[T1]{fontenc}    
\usepackage{hyperref}       
\usepackage{xcolor}       
\usepackage{amsthm}
\usepackage{algorithm}
\usepackage{algpseudocode}
\usepackage{url}            
\usepackage{booktabs}       
\usepackage{amsfonts}       
\usepackage{amsmath}
\usepackage{nicefrac}       
\usepackage{microtype}      
\usepackage{lipsum}
\usepackage{fancyhdr}       
\usepackage{graphicx}       
\graphicspath{{media/}}     

\DeclareMathOperator*{\argmin}{arg\,min}
\def\blambda{\boldsymbol{\lambda}}
\def\bd{\boldsymbol{d}}
\newtheorem{theorem}{Theorem}[section]
\title{Ensemble-Based Estimation of Alzheimer’s Disease Incidence from Dynamic Population Reconstructions 
}

\author{
  Giulia Bertaglia, Elisa Iacomini \\
  Department of Environmental and Prevention Sciences, \\
  University of Ferrara, \\
  Ferrara, Italy\\
  \texttt{\{giulia.bertaglia, elisa.iacomini\}@unife.it} \\
   \And
  Alex Viguerie \\
  Department of Pure and Applied Sciences
  Universit\`a degli Studi di Urbino Carlo Bo \\
   Urbino, PU (Italy)\\
  \texttt{alexander.viguerie@uniurb.it} \\
}

\begin{document}
\maketitle

\begin{abstract}
We present a two-stage methodology for reconstructing Alzheimer's disease (AD) incidence over time using ensemble Kalman inversion (EKI) applied to mortality data. In the first stage, we use EKI to infer temporal trends in all-cause and Alzheimer’s-specific mortality by fitting an age-structured demographic model to observed death counts. This yields posterior estimates of evolving population structure and age-specific AD mortality rates. In the second stage, we apply a back-calculation procedure that uses these estimates, along with the hazard of AD-related death following disease onset, to infer time- and age-specific incidence rates. This reverse-inference framework enables the reconstruction of latent disease dynamics in the absence of direct incidence surveillance. By integrating demographic structure, disease-specific hazards, and observed mortality into a coherent inferential pipeline, our approach offers a principled and flexible tool for monitoring chronic disease trends and estimating historical disease burden.

\end{abstract}


\section{Introduction}
Alzheimer’s disease (AD) represents one of the most pressing public health challenges of the 21st century. With aging populations across the globe, the burden of dementia-related mortality and disability is projected to rise dramatically over the coming decades, straining both formal healthcare systems and informal caregiving networks \cite{cimler2019predictions, wolters2020twenty, zissimopoulos2018impact}. Accurate, timely estimation of Alzheimer’s incidence across age and time is essential for forecasting disease burden, evaluating the effectiveness of preventive strategies, and guiding resource allocation at the national and subnational levels. Yet, despite its importance, direct estimation of AD incidence remains extremely difficult in practice \cite{brookmeyer2011national, tahami2022alzheimer, gustavsson2023global, alzheimer2025}.

Several factors complicate the direct measurement of AD incidence. These include delayed diagnosis, limited screening, variability in clinical definitions over time, and the absence of centralized or universal case registries in many regions \cite{gustavsson2023global, tahami2022alzheimer, monfared2024prevalence}. Longitudinal cohort studies can yield high-quality incidence estimates, but are expensive, slow, and subject to cohort effects that limit generalization \cite{joling2020time}. In contrast, AD-related mortality data—while a lagged and indirect signal—are often available with high temporal and demographic resolution in high-income countries, thanks to vital registration systems \cite{cdc_wonder,brown2024trends, romero2014under,stokes2020estimates}. This disparity motivates the use of back-calculation methods, in which current or recent mortality patterns are used to infer past incidence rates based on knowledge of disease progression and delay distributions \cite{egan2015review}.

The basic idea of back-calculation is well-established in epidemiology, having been originally developed and applied extensively to HIV/AIDS through the pioneering work of R. Brookmeyer and others \cite{brookmeyer1989statistical, brookmeyer1991reconstruction, song2017using, viguerie2023isolating, viguerie2024covid}, various cancers \cite{nautiyal2018spatiotemporal, ventura2014estimating, mariotto2017estimation}, and other chronic or progressive conditions where direct case ascertainment is difficult \cite{de2004estimating,trubnikov2014estimated}. In the AD context, back-calculation methods have been applied to cohort data \cite{yu2011estimating}; however, large-scale estimates of incidence and prevalence have tended to rely on model-based estimates or statistical extrapolation \cite{rajan2021population, brookmeyer1991reconstruction, rommel2025dementia,raggi2022incidence, monfared2024prevalence}.

In the present work, we propose novel back-calculation framework for reconstructing age- and time-specific incidence of Alzheimer’s disease in the United States from mortality data and dynamic reconstructions of the at-risk population. We model mortality as a convolution of past incidence, the age structure of the at-risk population, and the distribution of time from disease onset to death. The goal is to invert this mapping to recover the latent incidence function, structured by both age and calendar time. Our approach integrates three data sources:
\begin{itemize}
    \item Age-structured population distributions over time, which evolve under birth, death, and cohort aging dynamics, and are modeled via a demographic simulator based on a filtered partial differential equation (PDE)  using an Ensemble Kalman inversion (EKI) scheme;
\item AD-specific mortality counts, resolved by age and year, drawn from the Centers for Disease Control and Prevention (CDC) WONDER online database \cite{cdc_wonder};
\item A time-to-death distribution, describing the probabilistic lag between AD onset and death, calibrated from published cohort studies and approximated via a Weibull distribution with empirically supported shape and scale parameters \cite{joling2020time}.
\end{itemize}

Unlike classical back-calculation approaches that estimate incidence from deaths using stationary or time-series smoothing methods, our framework integrates population dynamics and mortality via an age-structured PDE, and formulates the incidence estimation problem as a structured inverse problem for a convolutional operator with time-varying demographic weights. Discretizing this equation over age and time yields a structured linear inverse problem, with the forward matrix encoding both the demographic dynamics and the convolution kernel implied by the onset-to-death distribution. Solving this inverse problem requires regularization, as the operator is compact and the data are noisy. We adopt Tikhonov regularization \cite{tikhonov} with a second-derivative (Laplacian) penalty over the age domain, which encourages smoothness in the inferred hazard functions while preserving temporal variation. This allows for stable, interpretable estimates of the age-specific incidence surface.

Furthermore, while classical back-calculation often relies on point estimates of the population at risk, we explicitly account for demographic uncertainty by propagating ensemble-based estimates of the susceptible population obtained from a dynamic PDE model. This coupling enables us to embed the epidemiological inverse problem within a broader Bayesian filtering framework, where uncertainty in demographic projections informs the uncertainty of reconstructed incidence. The core inverse solver remains deterministic, but operates on a posterior-informed input distribution, thereby blending stochastic population dynamics with a principled mechanistic structure for incidence-to-deaths mapping.

We demonstrate the application of the proposed methodology to reconstruct Alzheimer’s incidence in a realistic demographic setting, using data from \cite{cdc_wonder} over the period 2000-23. Our findings show that the method produces plausible, smooth incidence curves that are consistent with observed mortality trends and known features of Alzheimer’s natural history. Moreover, we provide an explicit quantification of age- and time-specific hazard rates, as well as annualized incidence estimates suitable for public health reporting.

This work aims to enhance the development of mathematically rigorous methods for the indirect inference of chronic disease dynamics, and highlights the value of combining epidemiological modeling, statistical inverse theory, and modern filtering techniques. To our knowledge, this is the first attempt to reconstruct time- and age-specific Alzheimer’s incidence from mortality using a fully data-informed convolutional backprojection method grounded in dynamical demography.

The remainder of the manuscript is organized as follows. In Section \ref{sec:model}, we present the demographic model along with the population reconstruction method based on the EKI. Section \ref{sec:incidence} describes the estimation of the incidence of AD. In Section \ref{sec:results}, we present and discuss the numerical application to U.S. data and the results obtained. Final conclusions are given in Section \ref{sec:conclusions}.

\section{Demographic Model and Population Reconstruction}\label{sec:model}

Let $u(a,t)>0$ be the number of persons aged $a$ at time $t$, $\mu(a,t)>0$ the mortality hazard for persons age $a$ at time $t$, $\alpha(t)>0$ the number of new births at time $t$, $u_0(a)$ the initial population age distribution, and $\xi(a,t)>0$ the number of immigrants aged $a$ at time $t$. Note that immigration must be explicitly accounted for in population modeling due to its substantial and lasting demographic impact. In particular, the magnitude of immigrant inflows is often of the same order as natural demographic changes and therefore cannot be considered negligible. For example, individuals who immigrated to the United States in their 20s or 30s around 2006 become part of the 40–50 age cohort by 2018, significantly influencing the structure of the population age distribution.

We consider an ensemble of populations aged $a$ at time $t$, $u_i(a,t)$, where $i=1,\,2,\,...,\,N_e$, $N_e$ number of ensembles, each evolving according to the following partial differential equation:

\begin{align}\begin{split}\label{demographicModel}
\dfrac{\partial u_i}{\partial t} + \dfrac{\partial u_i}{\partial a} &= -\mu_i(a,t)u_i + \xi_i(a,t), \quad t > t_0,\ a > 0, \\
u_i(0,t) &= \alpha(t), \\
u_i(a,t_0) &= u_{i,0}(a),
\end{split}\end{align}
where $t_0$ identifies the initial time.

We estimate the time-varying mortality hazard function $\mu_i(a,t)$ from observed, age-binned death counts using a Bayesian data assimilation technique known as the \textit{Ensemble Kalman inversion} (EKI). This approach treats $\mu(a,t)$ as a latent field and updates an ensemble of candidate solutions by assimilating observed age-structured death data through a Kalman-type correction. 

Let the discretized form of $\mu_i(a,t)$ over an age-time mesh be represented as a vector $\boldsymbol{\mu}_i \in \mathbb{R}^{n}$, and let $\mathbf{y}^{\text{obs}} \in \mathbb{R}^{m}$ be the vector of observed mortality counts across age bins and time points. The population model \eqref{demographicModel} defines a forward operator $\mathcal{H}(\boldsymbol{\mu}_i)$ which maps each candidate $\boldsymbol{\mu}_i$ to predicted observations (i.e., death counts) via numerical integration of the PDE system.

Each ensemble member’s initial population distribution, denoted by $u_{i,0}(a)$, is generated using age-binned United States (U.S.) Census population data for the year 2000. Specifically, we enforce that the total population within each discrete age bin (e.g., 5-year intervals) matches the observed data exactly across all ensemble members. However, within each bin, the age distribution is sampled uniformly at random for each ensemble member. This results in an ensemble of population trajectories that are all consistent with observed age-binned population totals, while allowing uncertainty in the finer intra-bin structure. In this way, the ensemble maintains statistical similarity to the empirical population while enabling uncertainty propagation throughout the modeling pipeline.

Annual births $\alpha(t)$ are defined using publicly available natality data \cite{cdc_wonder} and are the same for each ensemble member.  For the immigration term $\xi(a,t)$, we fit Weibull distributions for each year to match the data for the number and ages of US immigrations given in \cite{camarota_cis_2021}. For each ensemble member $u_i$, we sample the $\gamma_i(a,t)$ independently, to ensure proper propagation of the uncertain data in the ensembling. 

\subsection{EKI Algorithm for Mortality Estimation}

The Ensemble Kalman inversion is a powerful derivative-free method for solving inverse problems, particularly when data are noisy. By iteratively updating an ensemble of estimates, EKI effectively incorporates observational data to refine solutions without requiring gradient information. This makes it especially attractive for problems involving complex or computationally expensive forward models. EKI has recently gained significant attention due to its robustness in handling uncertainty and its adaptability to a wide range of applications, as highlighted in recent developments and theoretical insights such as \cite{herty2019kinetic,herty13recent,schillings2017analysis, herty2023filtering,Li2020} and references therein.
\par A first application of this approach to the reconstruction of mortality rates and disease incidence from limited and noisy data has been recently explored in \cite{viguerie2025aging}, demonstrating the potential of EKI in epidemiological modeling and public health contexts.

\par To define the age- and time-dependent mortality rates $\mu$, the EKI proceeds iteratively over time steps, updating each ensemble member via the following procedure:
\begin{algorithm}[H]
\caption{EKI for Mortality Estimation}
\begin{algorithmic}[1]
\State \textbf{Initialize} ensemble $\{ \boldsymbol{\mu}_i^{(0)} \}_{i=1}^{N_e}$ from a prior distribution
\For{each time step $t_j$}
    \For{each ensemble member $i = 1,\dots,N_e$}
        \State Simulate $u_i^{(j)}$ using $\mu_i^{(j-1)}$ in Eq.~\eqref{demographicModel}
        \State Compute predicted deaths: $\mathbf{y}_i^{(j)} = \mathcal{H}(\mu_i^{(j-1)})$
    \EndFor
    \State Compute sample mean $\bar{\mu}^{(j)}$, covariance $C_{\mu y}^{(j)}$, and $C_{yy}^{(j)}$
    \State Kalman gain: $K^{(j)} = C_{\mu y}^{(j)} \left(C_{yy}^{(j)} + R^{(j)}\right)^{-1}$
    \For{each ensemble member $i = 1,\dots,N_e$}
        \State Update:
        \[
        \mu_i^{(j)} = \mu_i^{(j-1)} + K^{(j)} \left( \mathbf{y}^{\text{obs},(j)} + \epsilon_i^{(j)} - \mathbf{y}_i^{(j)} \right)
        \]
        where $\epsilon_i^{(j)} \sim \mathcal{N}(0, R^{(j)})$
    \EndFor
\EndFor
\State \textbf{Return} posterior ensemble $\{ \mu_i^{(J)} \}_{i=1}^{N_e}$
\end{algorithmic}
\end{algorithm}

\par This procedure balances the prior dynamics imposed by the PDE system with observed data, namely, annual death counts in year $j$ grouped by discrete age-bins $\mathbf{y}^{\text{obs},j}$,  producing a temporally coherent, data-consistent reconstruction of $\mu(a,t)$. The ensemble structure allows for nonparametric uncertainty quantification and accommodates spatial smoothing across age in a natural manner.

\section{Alzheimer's incidence function}\label{sec:incidence}

After reconstructing the population ensemble through the previously-outlined procedure, we can use the members $u_i(a,t)$ to estimate the incidence of Alzheimer's disease in the population, $\lambda(a,t)$. For ease of explanation, we present the reconstruction process for a single member of the ensemble $u_i$; repeating this process for each $u_i$, $i=1,\,2,\,...,\,N_e$, allows us to obtain age-time distributions of the annual incidence of AD which naturally incorporate uncertainty, following a classical Monte Carlo approach \cite{caflisch1998}.

\par Let $\gamma(t)$ be a known probability density function that represents the time distribution from the onset of Alzheimer's to death. In the present, we use the definition
\begin{equation}\label{gammaDist}
\gamma(\tau) = \frac{k}{\vartheta} \left( \frac{\tau}{\vartheta} \right)^{k-1} e^{-(\tau/\vartheta)^k},
\end{equation}
with scale parameter $\vartheta\sim \mathcal{N}(6.3841,0.411) $ years and shape parameter $k\sim \mathcal{N}(1.4769, 0.0068)$. Note that $k>1$ in general, reflecting the increase in the probability of mortality over time. We independently sample $\vartheta_i$ and $k_i$ from for each ensemble element to account for uncertainty. Note that the parameter distributions were obtained through a nonlinear least-squares fitting following the survival data post-onset given in \cite{joling2020time}.

\par The number of Alzheimer’s deaths $d(a,t)$ at age $a$ and time $t$ is given by:
\begin{align}\label{deathKernel}
d(a,t) = \int_0^t \gamma(\tau)\, u(a-\tau, t-\tau)\, \lambda(a-\tau, t-\tau)\, d\tau.
\end{align}

We discretize the age domain into $K$ non-overlapping intervals:
\[
I_k = [a_k, a_{k+1}), \quad I = \bigcup_{k=0}^{K-1} I_k.
\]
Then, let $\{t_j\}_{j=0}^{J}$ be a sequence of time points. For each age bin $I_k$ and time $t_j$, define the aggregated deaths:
\begin{align}\label{discreteA}
d_{kj} = \int_{a_k}^{a_{k+1}} \int_0^{t_j} \gamma(\tau)\, u(a - \tau, t_j - \tau)\, \lambda(a - \tau, t_j - \tau)\, d\tau\, da.
\end{align}

We further approximate $\lambda(a,t)$ using a piecewise constant basis:
\[
\lambda(a,t) \approx \sum_{r=0}^{K'-1} \sum_{s=0}^{J'-1} \lambda_{rs}\, \phi_{rs}(a,t),
\]
where:
\[
\phi_{rs}(a,t) = \begin{cases}
1 & \text{if } a \in [a_r, a_{r+1}),\ t \in [t_s, t_{s+1}) \\
0 & \text{otherwise}.
\end{cases}
\]
Note that the intervals defined by $K'$ and $J'$ are, in general, different than those of the age-binned mortality data (provided by \cite{cdc_wonder}), as they are defined based on the discretization of the demographic model \eqref{demographicModel}. Substituting this into \eqref{discreteA} yields:
\begin{align}\label{discreteB}
d_{kj} = \sum_{r=0}^{K-1} \sum_{s=0}^{J-1} \lambda_{rs} \int_{a_k}^{a_{k+1}} \int_0^{t_j} \gamma(\tau)\, u(a - \tau, t_j - \tau)\, \phi_{rs}(a - \tau, t_j - \tau)\, d\tau\, da.
\end{align}

Let us define
\begin{align}
A_{kj, rs} := \int_{a_k}^{a_{k+1}} \int_0^{t_j} \gamma(\tau)\, u(a - \tau, t_j - \tau)\, \phi_{rs}(a - \tau, t_j - \tau)\, d\tau\, da,
\end{align}
so that the discretized system becomes:
\begin{align}\label{linearSys}
A \blambda = \bd,
\end{align}
where $\bd \in \mathbb{R}^{K\cdot J}$ is the vector of observed Alzheimer’s deaths and $\blambda \in \mathbb{R}^{K'\cdot J'}$ contains the incidence parameters to be estimated. 

\par Note that the matrix $A$ is generally \emph{not square}. The number of observed death counts $d_{kj}$ is determined by the resolution of available mortality data—typically aggregated over coarser age bins and annual time steps. In contrast, the basis used to approximate the incidence function $\lambda(a,t)$ can be defined on a much finer mesh, aligned with the age and time resolution of the underlying population structure $u(a,t)$. 

Let $K$ and $J$ denote the number of age and time bins used to report deaths, and let $K'$ and $J'$ denote the number of age and time bins used to define the basis functions $\phi_{rs}(a,t)$. Then the system \eqref{linearSys} has size $A \in \mathbb{R}^{(K \cdot J) \times (K' \cdot J')}$, and is undetermined in general, owing to the difference in age resolution of observed mortality (5-year age groups) and the age resolution of $u(a,t)$ (often finer than a year). 
The structure of \eqref{linearSys}, indeed, is such that, for a period of years $i=1,\,2,...,\,T$ :

\begin{equation}
    A=\begin{bmatrix}
    \boldsymbol{\Gamma}_1 \boldsymbol{U}_1 & 0 & 0 & ... & 0 \\ 
    \boldsymbol{\Gamma}_2 \boldsymbol{U}_1 & \boldsymbol{\Gamma}_1 \boldsymbol{U}_2 & 0 & ... & 0 \\     
    \boldsymbol{\Gamma}_3 \boldsymbol{U}_1 & \boldsymbol{\Gamma}_2 \boldsymbol{U}_2 & \boldsymbol{\Gamma}_1 \boldsymbol{U}_3 & ... & 0 \\     
    \vdots & \vdots & \vdots & \ddots & \vdots \\
    \boldsymbol{\Gamma}_T \boldsymbol{U}_1 & \boldsymbol{\Gamma}_{T-1} \boldsymbol{U}_2 & \boldsymbol{\Gamma}_{T-2} \boldsymbol{U}_3 & ... & \boldsymbol{\Gamma}_{1} \boldsymbol{U}_T \\     
    \end{bmatrix} ,   \blambda = \begin{bmatrix} \lambda_{1} \\ \lambda_2 \\ \lambda_3 \\ \vdots \\ \lambda_T \end{bmatrix}, \, \bd = \begin{bmatrix} d_{1} \\ d_2 \\ d_3 \\ \vdots \\ d_T \end{bmatrix},  \end{equation}
where the $\lambda_i$ correspond to the age-dependent Alzheimer's incidence during the year $i$, the $d_i$ Alzheimer's deaths in each age bracket in year $i$, $\boldsymbol{U}_i$ the population age structure in year $i$, and $\Gamma_i$ the probability of death $i$ years after Alzheimer's onset. As such, we expect smoothness within each $\lambda_i$ (the age domain), but not necessarily temporal smoothness across $\lambda_i$. A naive pseudo-inversion will not respect these conditions in general. We incorporate this prior belief by defining the matrix:
\begin{equation}
    L=\begin{bmatrix}
        2 & -1 & 0 & 0 & ... & 0 & 0 & 0 \\
        -1 & 2 & -1 & 0 & ... & 0 & 0 & 0 \\ 
        0 & -1 & 2 & -1 & ... & 0 & 0 & 0 \\
        \vdots&   \vdots & \vdots & \ddots & \ddots & \ddots & \vdots & \vdots \\        
        0 & 0 & 0 & 0 & ... & -1 & 2 & -1 \\
        0 & 0 & 0 & 0 & ... & 0 & -1 & 2
    \end{bmatrix}, \,\,\, L = C^T C, \,\,\, R = I \otimes C,
    \end{equation}
where the Cholesky factor $C$ is a lower triangular matrix whose existence is guaranteed by the symmetric positive definiteness of $L$. Note that
\begin{align}
    R^T R = (I \otimes C )^T (I \otimes C) = (I^T I ) \otimes (C^T C) =  I \otimes L,
\end{align}
ensuring that $L$ is applied independently to each $\lambda_i$.
We finally recover $\blambda$ by solving the following minimization problem:
\begin{align}\label{origSystem}
    \argmin_{\blambda} \{ \|A\blambda - \bd \|_2^2 + \beta \| R \blambda \|_2^2\}, 
\end{align}
where $\beta>0$ is a regularization parameter. In the current work, we chose $\beta=10^6$, based on an empirical L-curve analysis, shown in Fig. \ref{fig:tikhChoice} \cite{hansen1999curve}.

\begin{figure}[t]
    \centering
    \includegraphics[width=0.85\textwidth]{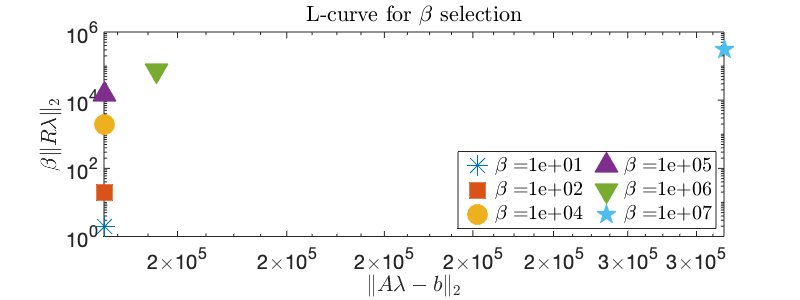}
    \caption{L-curve analysis for the choice of $\beta$. We find that $\beta=10^6$ provides the optimal tradeoff of accuracy and regularization.}
    \label{fig:tikhChoice}
\end{figure}

\begin{theorem}
Let $A\in \mathbb{R}^{m\times n}$, $d\in \mathbb{R}^m$, $R\in \mathbb{R}^{p\times n}$ and  $\beta>0$ the regularization parameter. Then the problem \eqref{origSystem} admits a unique solution:
        \begin{align}\label{uniqueSol}
\blambda= (A^T A + \beta R^T R )^{-1} A^T\bd.
\end{align}
\end{theorem}
\begin{proof}
    The objective function in \eqref{origSystem} is given by
    \begin{equation}
         J(\blambda)=\|A\blambda - \bd \|_2^2 + \beta \| R \blambda \|_2^2, 
    \end{equation}
    and can be rewritten as a quadratic form:
\begin{align}
    J(\blambda)&=(A\blambda-d)^T(A\blambda-d) + \beta(R\blambda)^T(R\blambda)\\
    &=\blambda^T(A^TA+\beta R^TR)\blambda - 2d^TA\blambda+d^Td.
\end{align}
The Hessian of $J(\blambda)$ is given by
\begin{equation*}
    \nabla^2J(\blambda)=2(A^TA-\beta R^TR).
\end{equation*}

By construction, $R^TR$ is positive definite, and, since $\beta>0$, the term $\beta R^TR$ is also positive definite. 
Adding the positive semi-definite matrix $A^TA$ to a positive definite matrix preserves positive definiteness. Therefore, $A^TA+\beta R^TR$ is positive definite, which implies that $\nabla^2J(\blambda)$ is positive definite. 
Hence, $J(\blambda)$ is strictly convex, which guarantees that the minimization problem has a unique global minimizer. 
\end{proof}

Notice that, in practice, the solution \eqref{uniqueSol} is equivalent to the solution of the following augmented least-squares problem:
\begin{align}\label{augLS}
    \argmin_{\blambda} \bigg\| \begin{bmatrix} A \\ \sqrt{\beta} R \end{bmatrix} \blambda - \begin{bmatrix}
        \bd \\ \boldsymbol{0}
    \end{bmatrix} \bigg\|_2^2.
\end{align}

\section{Results}
\label{sec:results}

We present here the primary outputs of the Alzheimer’s disease incidence reconstruction methodology, based on mortality-constrained back-calculation integrated with demographic simulation. These outputs include (i) the annual probability of developing Alzheimer’s disease (hazard), (ii) the total number of new cases per year (incidence), and (iii) the full spatiotemporal surface of Alzheimer’s hazard rates across age and calendar time. Together, they provide a rich, multiscale portrait of disease burden across the United States adult population over the past two decades. 

We highlight that, due to the COVID-19 pandemic, Alzheimer's mortality during the year 2020 was significantly elevated compared to previous and future years; particularly among deaths for which Alzheimer's was listed as a contributing factor \cite{alzheimer2025,cdc_wonder}. Based on these considerations, we applied a correction factor and reduced Alzheimer's deaths in 2020 by 15\% to avoid overestimating incidence, based on the excess of mortality observed in the considered cohort \cite{alzheimer2025}.  

We note that, for each year from 2005 to 2022, we provide full age-specific incidence and hazard curves, along with associated 99\% confidence intervals, in Appendix A. These reconstructions are derived from ensemble back-calculations and reflect uncertainty due to both stochastic sampling and data limitations. While these year-specific curves offer detailed insights, we focus here on summarizing the broader temporal and demographic trends that emerge across the study period. Our primary emphasis is on the evolution of total incidence over time, as well as how incidence has varied across key age strata. This approach allows us to highlight high-level patterns that persist across multiple years and ensemble members, providing a stable foundation for interpretation even where individual-year estimates may be more sensitive to noise or reporting artifacts.

We organize our results into three broad categories, each presented in its own subsection below:
\begin{itemize}
\item \textbf{Aggregate AD incidence} in the United States population from 2005 to 2022;
\item \textbf{Age-related AD risk}, capturing how hazard rates evolve by age across the study period;
\item \textbf{Age-specific AD incidence}, highlighting trends in the absolute burden of disease within distinct age groups.
\end{itemize}

This structure allows us to first characterize the total burden of Alzheimer’s disease over time, then investigate how risk evolves with age, and finally examine how different segments of the population have contributed to the observed shifts in incidence.

\subsection{Overall Alzheimer's Incidence}
In Fig. \ref{fig:incByYearOverall}, we depict annual AD incidence in the United States over the years 2005-22, with the year-by-year incidence estimates and 99\% confidence intervals provided in Table \ref{tab:annualIncidenceCounts}. Over the period 2005-12, we observe fairly steady AD incidence, around 400,000–600,000 cases per year. The years 2013 and 2014 showed notable growth in AD incidence, with more than 1 million new cases estimated in 2014 alone. This effect appears robust to variation in both population structure and delay kernel $\gamma$. However, they may also reflect improvements in diagnostic capacity and/or changes to attribution policy of AD on death certificates \cite{national2023physicians}. We also remark that the confidence intervals around our estimates for these years are considerably larger than in the years preceding and following, reflecting broader uncertainty.
\par Incidence remained stable over the period 2015-20, and, while lower than 2013-14 levels, notably higher compared to pre-2012, averaging around 600,000 new cases per year. The years 2021-22, in contrast, show a decline from 2019-20 levels. Similar trends have been found elsewhere \cite{freedman2024dementia, alzheimer2025}, and may reflect pruning effects related to the COVID-19 pandemic, in which excess mortality among elderly individuals during the pandemic years has led to fewer persons developing AD downstream. We also note, however, that these years are closer to the end of our data period, and are likely more subject to statistical artifacts. As such, the apparent decrease in AD in 2021-22 should be interpreted with caution.

\begin{figure}[ht]
    \centering
    \includegraphics[width=0.85\textwidth]{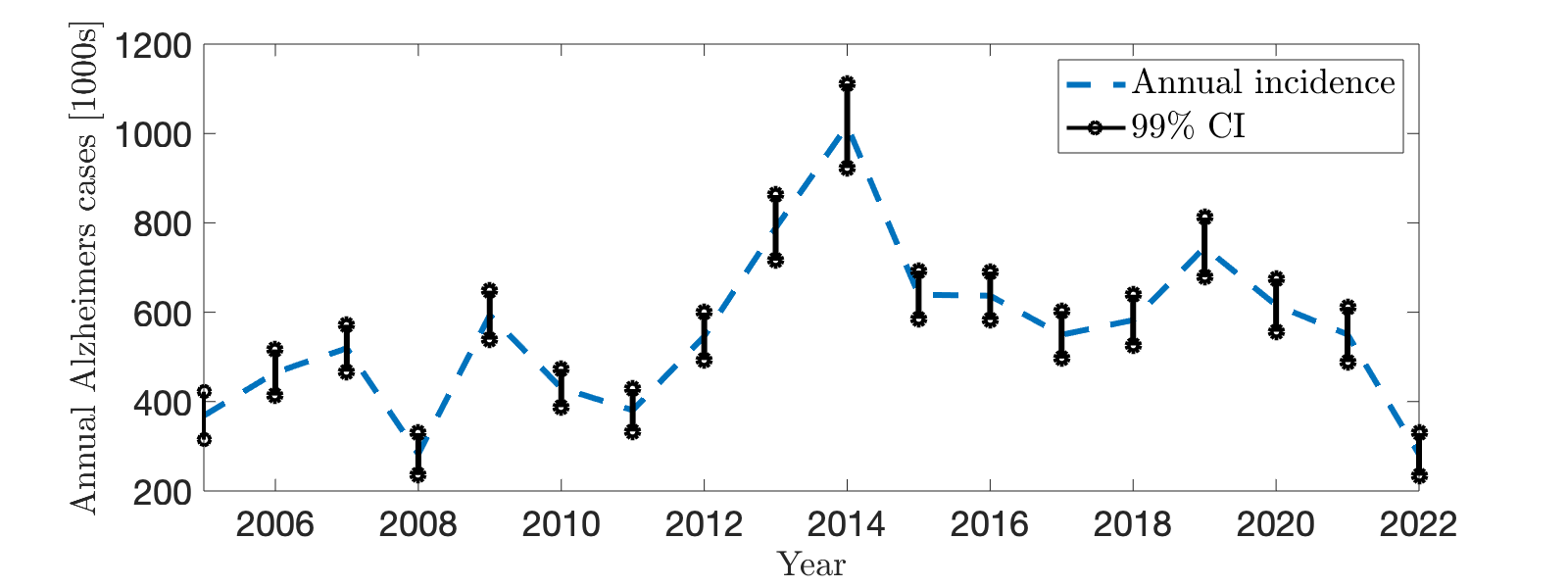}
    \caption{Estimated new Alzheimer's cases, overall, in the United States, 2005-22.}
    \label{fig:incByYearOverall}
\end{figure}

\begin{table}[ht]
\centering
\begin{tabular}{|l||c|c|c|}
\hline
Year & Incidence (1000s) & Lower 99\% CI & Upper 99\% CI \\ \hline \hline
2005 & 368.5 & 314.8 & 422.2 \\ \hline
2006 & 465.0 & 413.5 & 516.6 \\ \hline
2007 & 519.4 & 466.9 & 571.9 \\ \hline
2008 & 282.6 & 235.7 & 329.5 \\ \hline
2009 & 593.4 & 538.7 & 648.1 \\ \hline
2010 & 430.4 & 387.5 & 473.3 \\ \hline
2011 & 380.6 & 333.1 & 428.1 \\ \hline
2012 & 546.2 & 493.1 & 599.2 \\ \hline
2013 & 788.9 & 715.4 & 862.4 \\ \hline
2014 & 1017.3 & 922.6 & 1112.1 \\ \hline
2015 & 639.2 & 585.3 & 693.1 \\ \hline
2016 & 636.9 & 582.8 & 691.1 \\ \hline
2017 & 549.6 & 496.4 & 602.7 \\ \hline
2018 & 582.4 & 525.6 & 639.1 \\ \hline
2019 & 745.1 & 678.4 & 811.9 \\ \hline
2020 & 615.9 & 556.4 & 675.3 \\ \hline
2021 & 550.2 & 488.3 & 612.1 \\ \hline
2022 & 281.9 & 234.1 & 329.6 \\ \hline
\end{tabular}
\caption{Annual Alzheimer's incidence point estimates (in 1000s) and 99\% confidence intervals, in the United States 2005–2022.}
\label{tab:annualIncidenceCounts}
\end{table}

\subsection{Age-Specific Hazard of Alzheimer’s Onset}

Figure~\ref{fig:hazardByYear} shows the annual probability of developing Alzheimer’s disease by age (from age 60 to 90), for selected years. This quantity corresponds to the conditional probability of disease onset within one year for individuals of a given age who have not yet developed AD. A surface plot, showing age, time, and AD risk from ages 60-90 over the years 2005-22 is provided in Figure \ref{fig:hazardByYearSurf}.

Across all displayed years, the probability of AD onset increases exponentially with age from around age 60 through the early 80s, eventually flattening or declining slightly past age 85. The curves exhibit tight clustering across years, indicating consistency and temporal stability of the age-risk relationship. We observe a possible decrease in age-specific risk around the ages 70 to 85, plotted in Figure \ref{fig:hazardByYear70to85}. Later years (2015, 2017, and 2021) show generally lower risk than earlier years (2007, 2009, and 2013). However, we note that this trend is not uniform: 2011 shows a lower risk compared to later years, and 2019 shows a higher risk.  A decrease in age-specific risk for this age cohort, starting in the mid-2010s, has been observed in other studies \cite{freedman2024dementia}.

In Fig. \ref{fig:hazardByYear60to70}, we plot the AD risk hazard for ages 55-70 from 2005-22. While noisy, the reconstructions suggest increased AD risk at younger ages over time. This suggests that increases in AD diagnoses among persons under 70 may also be related to changes in risk, and not solely due to demographic factors. We stress, however, that these trends are non-monotonic and should be interpreted with caution.

The hazard curves are smooth and realistic, with no artifacts of overfitting or instability. Importantly, they respect known biological features of Alzheimer’s: slow buildup of risk before age 70, rapid increase through age 80, and tapering at the highest ages, where survivor bias and diagnostic limitations reduce apparent hazard.

\begin{figure}[t]
    \centering
    \includegraphics[width=0.85\textwidth]{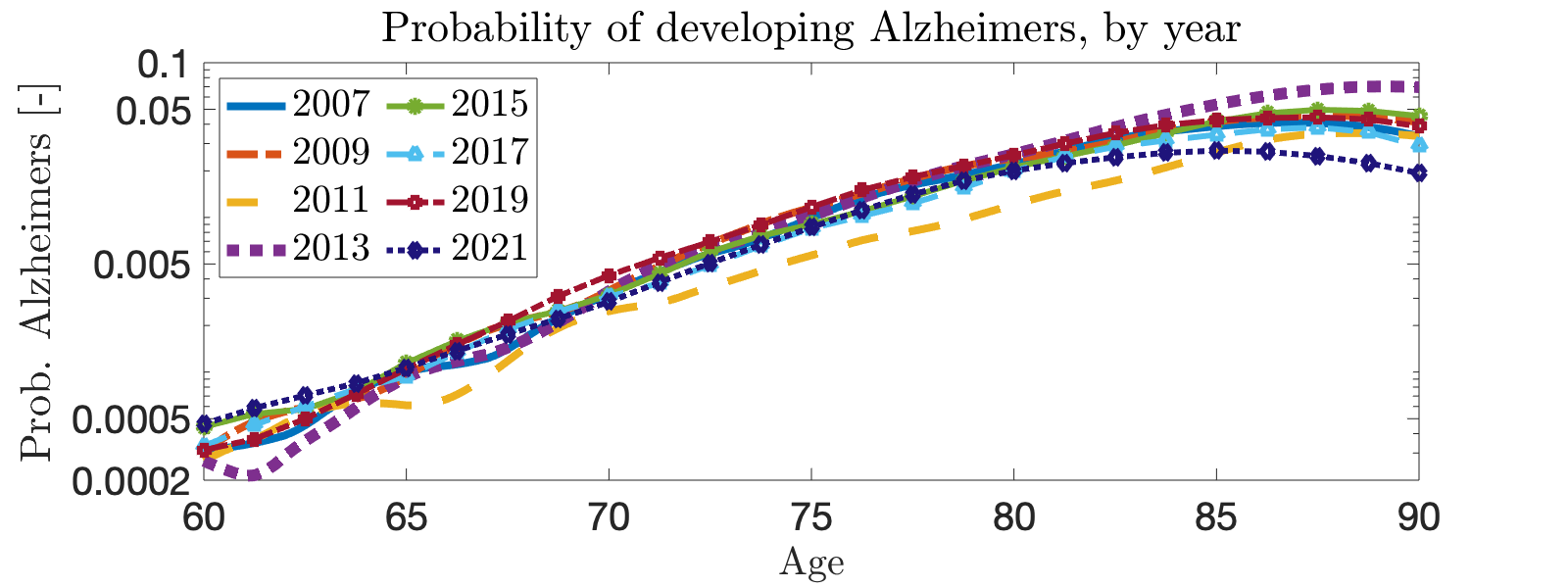}
    \caption{Annual probability of developing Alzheimer’s disease by age, for selected years (2007–2021). Risk increases exponentially with age, stabilizing beyond age 85. We observe relatively stable age-dependent risk in time; however, there is some evidence of increasing risk at younger ages.}
    \label{fig:hazardByYear}
\end{figure}

\begin{figure}[t]
    \centering
    \includegraphics[width=0.85\textwidth]{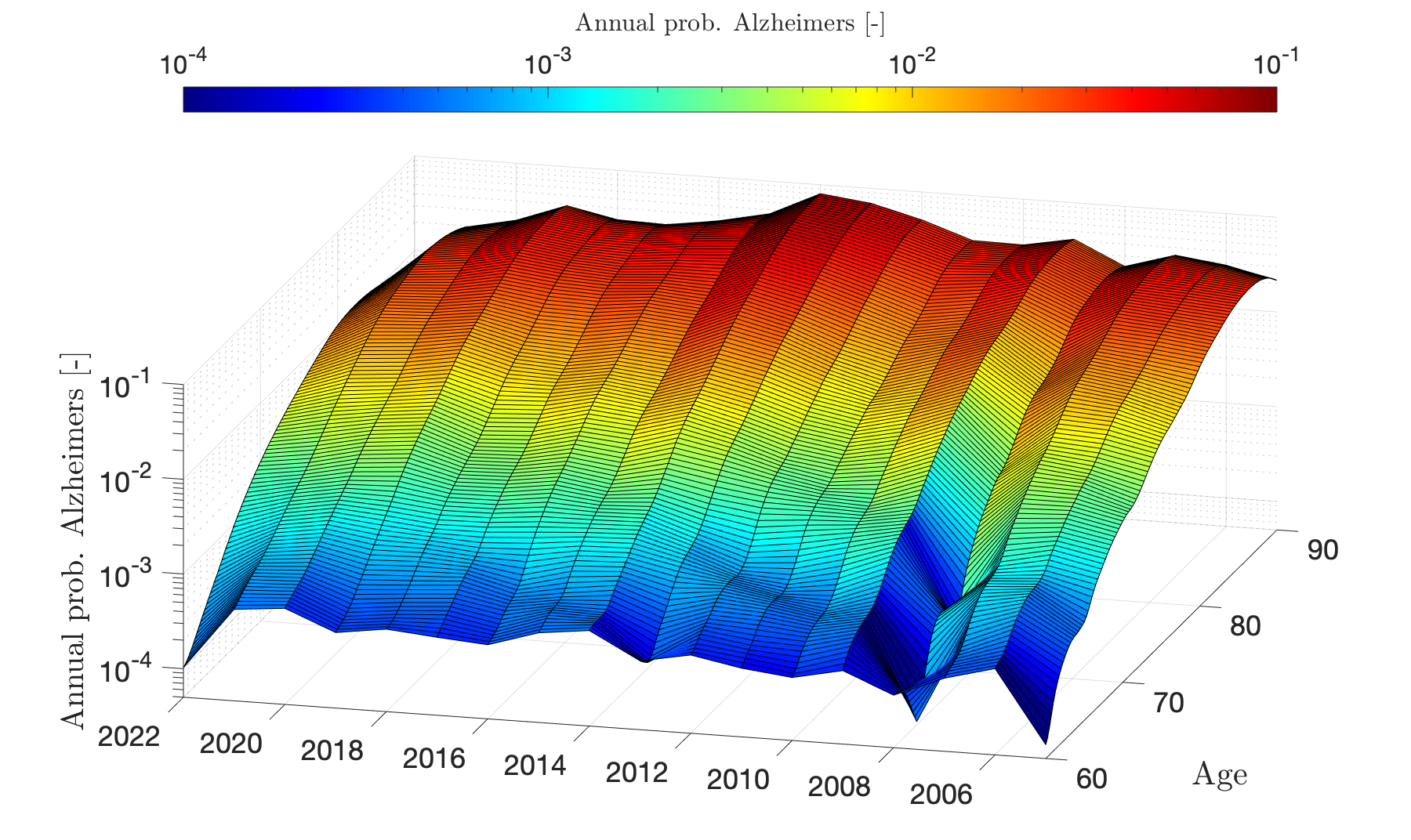}
    \caption{Surface showing Alzheimer's Hazards from 2005-22.}
    \label{fig:hazardByYearSurf}
\end{figure}

\begin{figure}[h]
    \centering
    \includegraphics[width=0.85\textwidth]{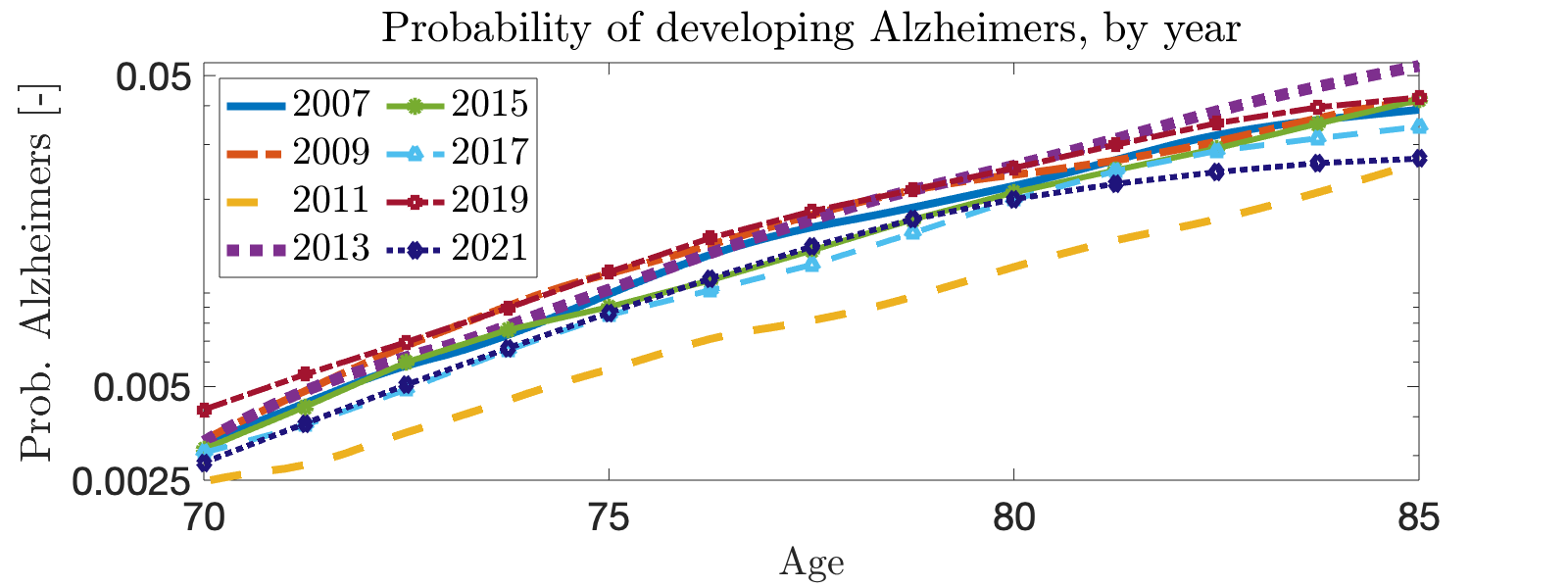}
    \caption{Annual probability of developing Alzheimer’s disease by age, for selected years (2007–2021) for ages 70 to 85. Risk in this age range appears relatively stable in time.}
    \label{fig:hazardByYear70to85}
\end{figure}

\begin{figure}[h]
    \centering
    \includegraphics[width=0.85\textwidth]{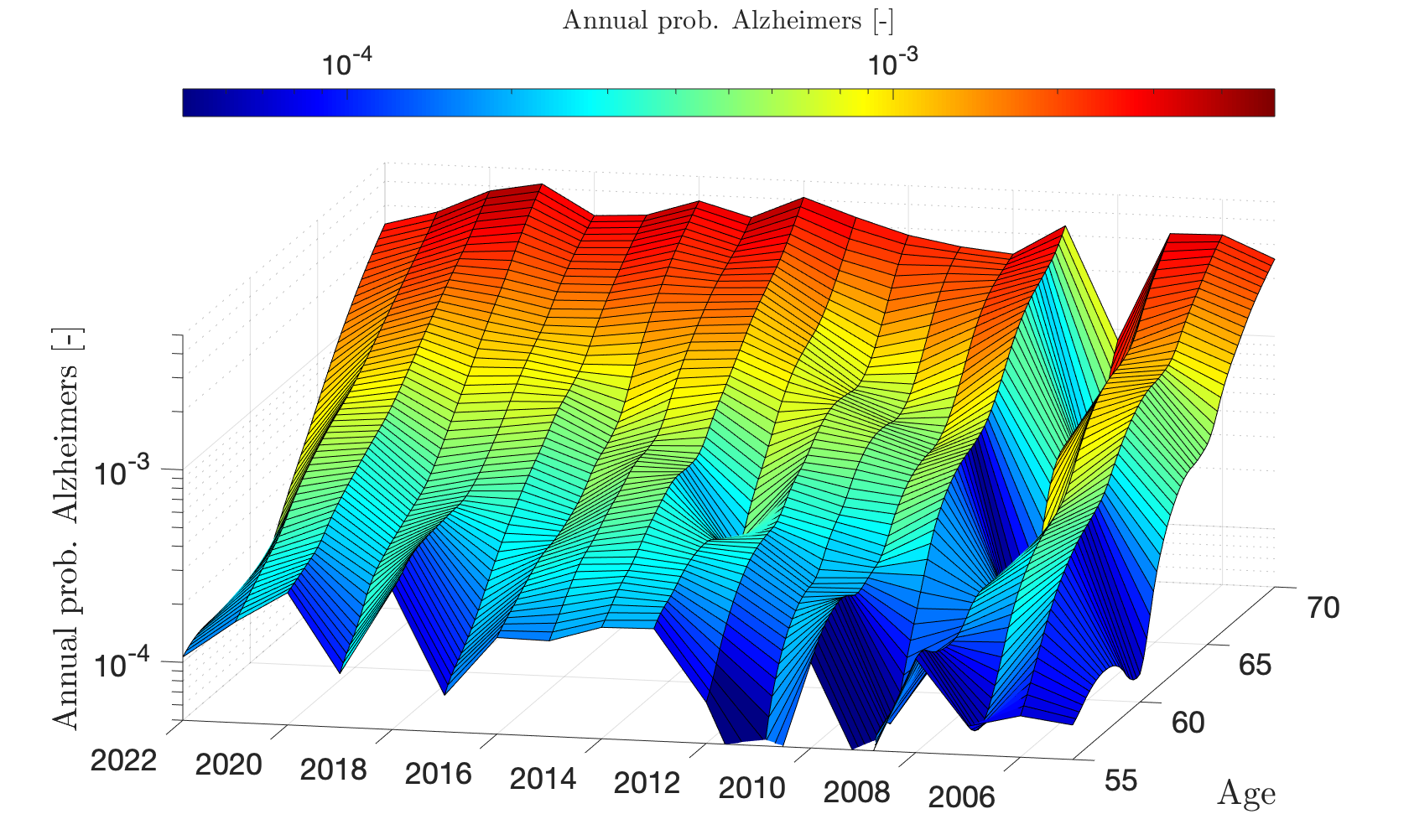}
    \caption{Annual probability of developing Alzheimer’s disease at younger ages (55-70) for selected years, 2005-22. Though volatile, the data suggest that the risk of developing AD in this age range may be increasing over time.}
    \label{fig:hazardByYear60to70}
\end{figure}

\subsection{Incidence: Total New Cases per Year}

Figure~\ref{fig:incByYearAge} shows the corresponding estimates of the total number of new Alzheimer’s cases per year, stratified by age, for the same years as above. This quantity is computed as the product of the hazard rate and the age-specific population size reconstructed via our demographic PDE model \eqref{demographicModel}:
\[
\text{Incidence}(a, t) = u(a, t) \cdot \lambda(a, t).
\]

The incidence curves reveal the joint effect of rising age-specific risk and population size. Incidence increases with age until roughly 82–86 years old, after which it begins to decline—reflecting the reduced population size at the oldest ages despite elevated hazard. The peak of the incidence curve thus emerges from a balance between demographic availability and risk intensity.

\begin{figure}[h!]
    \centering
    \includegraphics[width=0.85\textwidth]{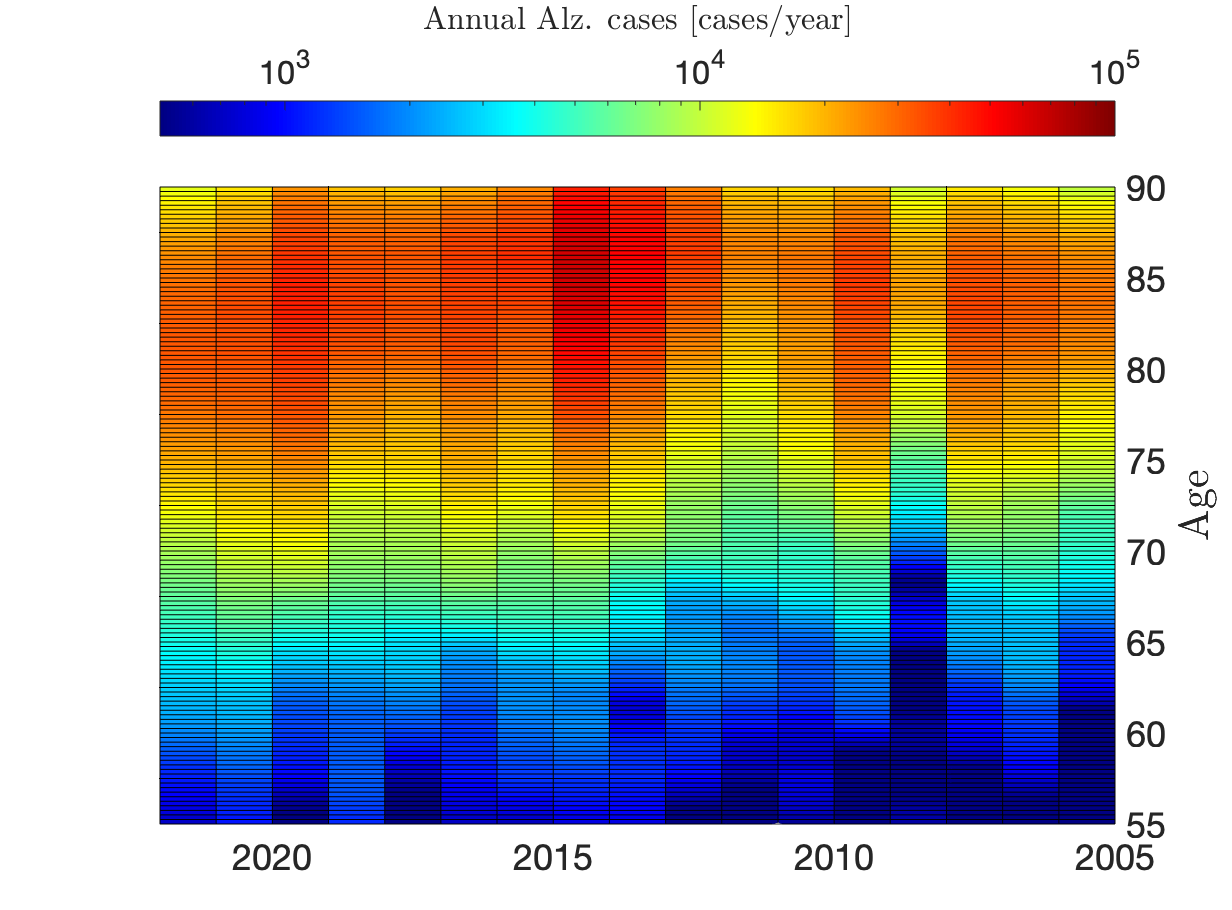}
    \caption{Estimated new Alzheimer’s cases per year by age, for selected years (2005–2022). Incidence peaks between ages 82–86 and declines thereafter due to reduced population at older ages. Later years show clear increases in total case burden at younger ages (75 and under), and decreased burden at older ages (85 and over).}
    \label{fig:incByYearAge}
\end{figure}

\begin{figure}[h!]
    \centering
    \includegraphics[width=0.85\textwidth]{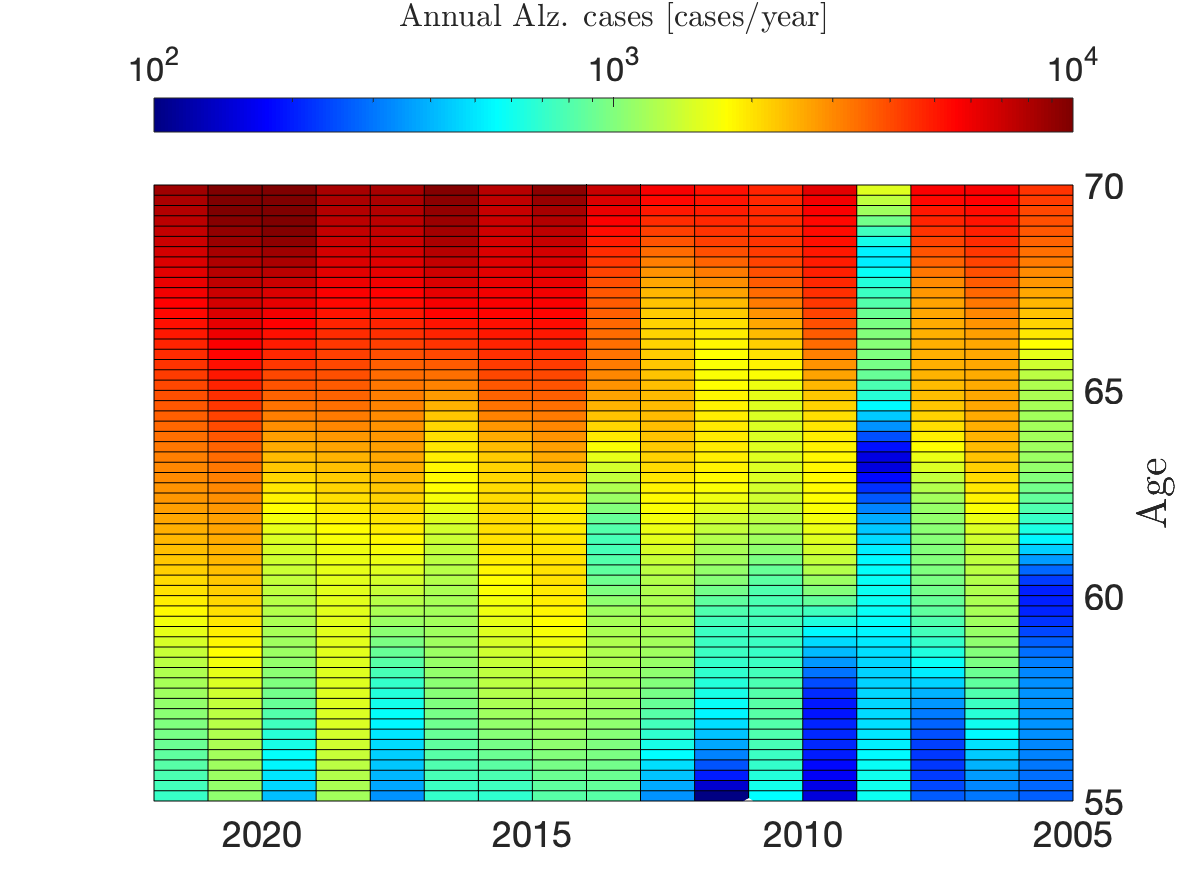}
    \caption{Age-specific AD incidence in the US, ages 55-70, years 2005-22. We observe a clear trend of increased incidence in this cohort over time.}
    \label{fig:incByYearAge_55to70}
\end{figure}

Compared to the hazard curves, inter-year differences are more pronounced. Later years (e.g., 2015, 2018, 2021) show higher total incidence across nearly all age bands, driven by population aging and increased absolute numbers of older adults. This is most clearly visible in the younger age ranges (75 and under). In contrast, later years (2018 and 2021) suggest that AD incidence in the 85 and older cohort may be decreasing somewhat. We plot the age-structured AD incidence across all ages 55-90 and all years 2005-22, in Fig. \ref{fig:incByYearAge}, showing that these trends hold more broadly. 

To better visualize the trend, in Fig. \ref{fig:incByYearAge_55to70} we plot a heatmap of annual AD incidence in the United States over the 55-70 age range in the years 2005-22. As suggested in the previous plots, AD incidence among persons under 70 shows a clear increasing trend in time. While demographic factors surely play a large role, as shown in Fig. \ref{fig:hazardByYear60to70}, this trend may also reflect increased risk. 

\begin{figure}[t!]
    \centering
    \includegraphics[width=0.85\textwidth]{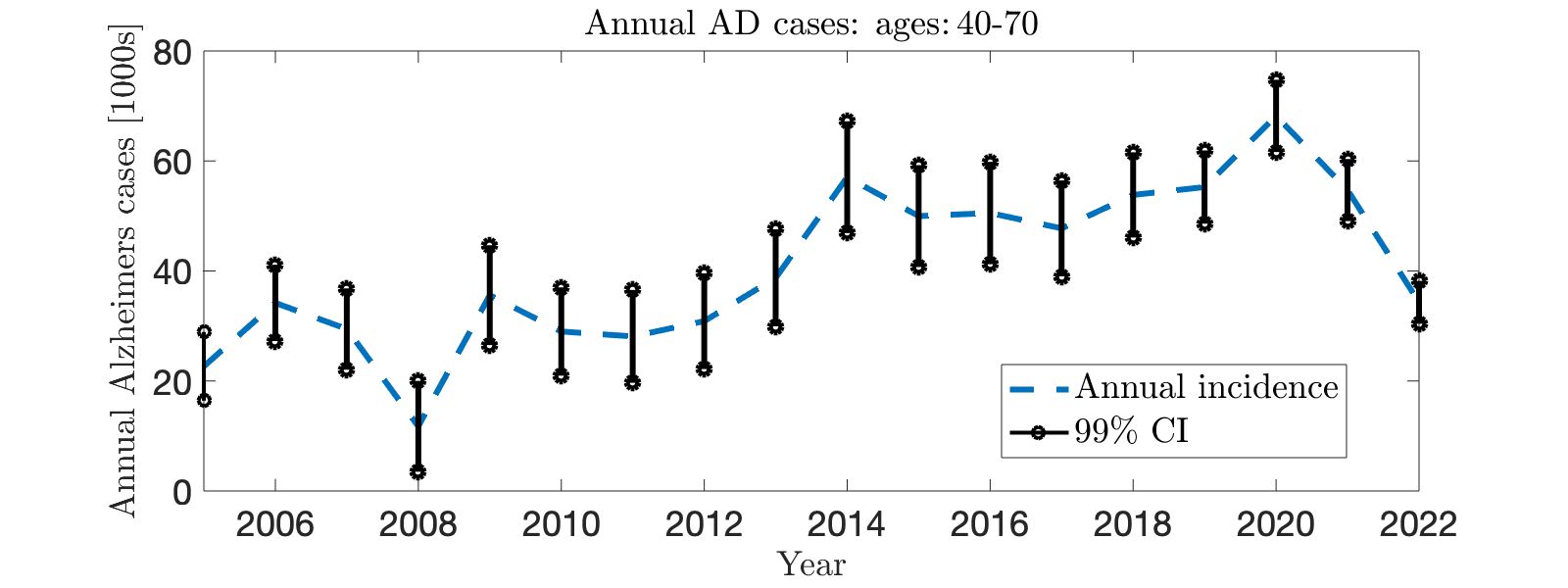}
    \caption{Total AD incidence among persons under 70 in the United States. Our reconstructions indicate a marked increase in incidence in this age group in recent years.}
    \label{fig:totalIncUnder70}
\end{figure}

\begin{table}[t]
\centering
\begin{tabular}{|l||c|c|c|}
\hline
Year & Incidence (1000s) & Lower 99\% CI & Upper 99\% CI \\ \hline \hline
2005 & 22.65 & 13.98 & 31.33 \\ \hline
2006 & 34.20 & 24.56 & 43.83 \\ \hline
2007 & 29.44 & 19.19 & 39.70 \\ \hline
2008 & 11.75 &  0.28 & 23.22 \\ \hline
2009 & 35.59 & 22.99 & 48.19 \\ \hline
2010 & 29.00 & 17.85 & 40.14 \\ \hline
2011 & 28.10 & 16.38 & 39.82 \\ \hline
2012 & 30.87 & 18.78 & 42.96 \\ \hline
2013 & 38.75 & 26.51 & 51.00 \\ \hline
2014 & 57.05 & 42.94 & 71.17 \\ \hline
2015 & 49.98 & 37.11 & 62.84 \\ \hline
2016 & 50.53 & 37.60 & 63.46 \\ \hline
2017 & 47.74 & 35.63 & 59.86 \\ \hline
2018 & 53.82 & 43.06 & 64.58 \\ \hline
2019 & 55.25 & 46.00 & 64.50 \\ \hline
2020 & 68.20 & 59.08 & 77.32 \\ \hline
2021 & 54.71 & 46.87 & 62.56 \\ \hline
2022 & 34.30 & 28.89 & 39.71 \\ \hline
\end{tabular}
\caption{Estimated annual Alzheimer’s incidence among persons under age 70 (in 1000s), with 99\% confidence intervals, United States 2005–2022.}
\label{tab:under70Incidence}
\end{table}

In Table \ref{tab:under70Incidence}, we report annual estimated AD incidence among individuals under age 70, providing further support for the age-specific trends discussed previously, with a corresponding plot provided in Fig. \ref{fig:totalIncUnder70}. Except for 2022—a more recent and likely less reliable estimate—we observe a clear upward trend in AD incidence within this age group. From 2005 to 2013, annual incidence generally ranged between 20,000 and 40,000 cases. In contrast, from 2018 to 2021, incidence levels are consistently in the 50,000–60,000 range. Notably, while overall AD incidence peaked in 2013–2014 (with 2014 standing out as a clear maximum), this pattern is not mirrored in the under-70 population. Instead, our reconstructions indicate that incidence in this younger group remained comparable to, or higher than, 2014 levels throughout the following years. 

\begin{figure}[t!]
    \centering
    \includegraphics[width=0.85\textwidth]{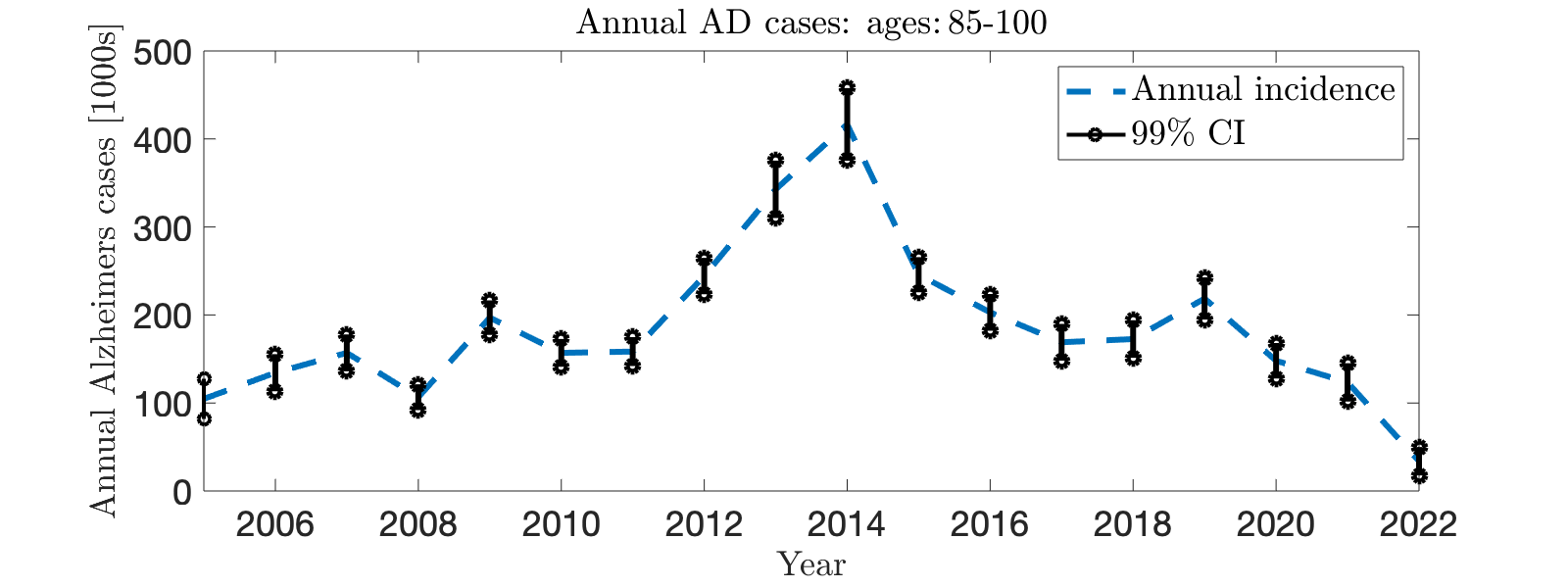}
    \caption{Total AD incidence among persons 85 and older in the United States. Our reconstructions indicate a marked decrease in incidence in this group in recent years, as compared to the mid-2010s peak.}
    \label{fig:totalIncOver85}
\end{figure}

\begin{table}[t]
\centering
\begin{tabular}{|l||c|c|c|}
\hline
Year & Incidence (1000s) & Lower 99\% CI & Upper 99\% CI \\ \hline \hline
2005 & 104.7 & 81.9 & 127.5 \\ \hline
2006 & 134.4 & 113.3 & 155.4 \\ \hline
2007 & 157.0 & 136.0 & 177.9 \\ \hline
2008 & 106.6 & 91.9 & 121.2 \\ \hline
2009 & 197.0 & 177.4 & 216.5 \\ \hline
2010 & 156.9 & 141.1 & 172.8 \\ \hline
2011 & 158.4 & 142.0 & 174.8 \\ \hline
2012 & 244.2 & 223.6 & 264.9 \\ \hline
2013 & 342.7 & 310.1 & 375.4 \\ \hline
2014 & 416.6 & 375.3 & 457.9 \\ \hline
2015 & 245.7 & 225.8 & 265.7 \\ \hline
2016 & 202.9 & 182.1 & 223.8 \\ \hline
2017 & 169.0 & 148.1 & 189.9 \\ \hline
2018 & 172.6 & 151.2 & 194.0 \\ \hline
2019 & 218.7 & 194.7 & 242.7 \\ \hline
2020 & 148.0 & 127.8 & 168.1 \\ \hline
2021 & 123.4 & 101.6 & 145.2 \\ \hline
2022 & 33.8 & 17.7 & 50.0 \\ \hline
\end{tabular}
\caption{Annual Alzheimer's incidence estimates among individuals aged 85 and older, with 99\% confidence intervals (in 1000s), United States 2005–2022.}
\label{tab:over85Incidence}
\end{table}

In Table \ref{tab:over85Incidence}, we report annual estimated AD incidence among individuals aged 85 and older (see also Fig \ref{fig:totalIncOver85}). In contrast to the upward trend observed among those under 70, we find a marked decline in AD incidence among the oldest adults over the past decade. From 2009 through 2014, annual incidence in this group rose sharply, reaching a peak of over 400,000 cases in 2014—substantially higher than any other year. However, following this peak, incidence declined steadily, falling below 250,000 cases by 2015 and dropping to nearly 120,000 by 2021. The estimate for 2022 is substantially lower still, though, as discussed previously, later-year estimates may be less reliable due to right-censoring and data limitations. Notably, this downward trend stands in contrast to the overall trajectory of total AD burden. Furthermore, this cannot be easily explained by demographic trends, as the overall number of persons 85 and older increased approximately 11\% from 2014 to 2022. In summary, these findings suggest that the apparent stability or growth in national AD incidence may be increasingly driven by younger age groups.

\section{Conclusion}\label{sec:conclusions}
In this work, we developed a novel workflow for estimating Alzheimer's disease incidence over time. We first used data assimilation to reconstruct a functional representation of the United States population $u(a,t)$ in both age- and time-domains. Then, using the reconstructed $u(a,t)$ and available age-structured AD mortality data, we developed a back-calculation algorithm to estimate age- and time-dependent AD risk and incidence. Furthermore, we rigorously established the existence and uniqueness of our inverse problem formulation.

\par Our results showed that age-dependent AD risk has remained relatively stable over the study period, particularly among ages 70-85. However, we found evidence of increasing risk among persons under 70, and decreasing risk among persons 85 and older. 

\par While the temporal trends in AD risk observed were quite small, we found evidence of clear temporal trends in AD incidence. Our analysis showed that overall AD incidence peaked in 2013-14, and while we observed lower incidence in the years following, the levels observed 2015-22 are markedly higher than 2005-12 levels. Additionally, AD incidence among persons aged 55-70 has increased sharply in recent years, exceeding even 2013-14 peaks, driven by demographic effects and potentially increasing AD risk.

\par This study has several important limitations. First, our survival function $\gamma$ does not incorporate additional age dependence, due to a lack of sufficiently granular data. In reality, survival following AD onset likely varies with age, not just time since diagnosis, and this simplification may introduce bias in age-structured incidence estimates. Second, our reconstructions rely on death certificate data as a proxy for AD incidence, which is known to under-report true AD burden and may be affected by changes in diagnostic or coding practices over time \cite{brown2024trends, romero2014under,stokes2020estimates}. While we have attempted to adjust for known artifacts—such as excess mortality during the COVID-19 pandemic—these corrections are inevitably approximate. However, we stress that such limitations primarily affect the \textit{quantitative} scale of our estimates (i.e., how many cases are inferred). Hence, even if the absolute level of AD-related mortality is systematically over- or under-counted, this is unlikely to reverse or obscure the major temporal and age-specific trends we observe. Thus, while caution is warranted in interpreting exact incidence counts, the broader, qualitative patterns revealed in our analysis are likely to remain valid across plausible levels of misreporting.

\par There is considerable room for further development of this work. More sophisticated assumptions on the survival function $\gamma$, including the incorporation of age- and/or time-dependence, should be explored. While the current work considered the entire U.S. population, stratification analyses that consider populations grouped by relevant demographic characteristics, including local or regional populations, may provide more detailed information useful for policymakers. Finally, while we have limited the scope of the present analysis to reconstruction and estimation of past incidence, we note that our framework can be used to project future incidence and risk trends as well, following approaches analogous to those shown in \cite{viguerie2025aging}. Such forecasts may be useful for policymakers in public health. Additionally, further investigation of numerical techniques, including more sophisticated, data-driven selection of regularization parameters \cite{Sgattoni2025}, may be worthwhile.

\section*{Acknowledgments}

All the authors are members of the Italian "National Group for Scientific Computation" (GNCS-INDAM), whose support is acknowledged.
The authors would also like to acknowledge Ruiguang Song, Paolo Cerasoli, and Siobhan O'Connor for their helpful input and suggestions.

\bibliographystyle{unsrt}  
\bibliography{references}

\appendix
\section*{Appendix A: Age-Specific AD Risk and Incidence by Year}
\addcontentsline{toc}{section}{Appendix A: Age-Specific AD Risk and Incidence by Year}

\renewcommand{\thefigure}{A\arabic{figure}}
\setcounter{figure}{0}

This appendix displays the estimated annual Alzheimer’s disease (AD) risk and incidence across age for each year from 2005 to 2022. Each pair of pages shows the annual hazard (probability of AD onset) and corresponding incidence (total new AD cases), stratified by age.

\clearpage

\foreach \year in {2005,...,2022} {
    
    \begin{figure}[p]
        \centering
        \begin{minipage}{0.48\textwidth}
            \centering
            \includegraphics[width=\linewidth]{uqByYear/hazard_\year.png}
        \end{minipage}
        \hfill
        \begin{minipage}{0.48\textwidth}
            \centering
            \includegraphics[width=\linewidth]{uqByYear/incidence_\year.png}
        \end{minipage}
        \caption{Age-specific Alzheimer’s disease hazard (left) and incidence (right) with 99\% confidence intervals, year \year.}
    \end{figure}

}

\end{document}